\documentclass{article}[12pt]
\usepackage[utf8]{inputenc}

\usepackage{mathtools}
\usepackage{hyperref}
\usepackage{wrapfig}
\usepackage{pgf,tikz}
\usetikzlibrary{arrows}
\usepackage{listings} 
\usepackage{tikz-cd}
\usepackage{amsmath,amsthm,amsfonts,amssymb, fancyhdr,graphics,enumerate,color,titlesec, xfrac}

    \AtEndDocument{
  \par
  \medskip
  \begin{tabular}{@{}l@{}}%
  \textsc{Andoni Zozaya} \\
    \textsc{Department of Mathematics} \\
     \textsc{University of the Basque Country UPV/EHU}  \\ \textsc{48080 Bilbao, Spain}\\
    \textit{E-mail address}: \texttt{andoni.zozaya@ehu.eus}
  \end{tabular}}

\theoremstyle{plain}

\newtheorem{lemma}{Lemma}[section]
\newtheorem{proposition}[lemma]{Proposition}

\newtheorem{corollary}[lemma]{Corollary}

\newtheorem*{subclaim*}{Subclaim}
\newtheorem*{mainthm*}{\normalfont\scshape Main Theorem}

\theoremstyle{remark}

\newtheorem*{remark*}{Remark}

\newtheorem*{notation*}{Notation}
\newtheorem*{question*}{Question}

\theoremstyle{definition}

\newtheorem{definition}[lemma]{Definition}

\newtheorem*{definition*}{Definition}
\newtheorem*{remarks*}{Remarks}
\newtheorem*{example*}{Example}
\newtheorem*{examples*}{Examples}
\newtheorem*{conjecture*}{Conjecture}
\newtheorem*{conjectures*}{Conjectures}

\newcommand\blfootnote[1]{%
  \begingroup
  \renewcommand\thefootnote{}\footnote{#1}%
  \addtocounter{footnote}{-1}%
  \endgroup
}

\def\N{\mathbb{N}}
\def\Z{\mathbb{Z}}
\def\F{\mathbb{F}}

\def\CC{\mathcal{C}}

\def\I{\textbf{I}}
\def\FL{\textbf{F}}
\def\G{\textbf{G}}
\def\W{\mathbf{W}}
\def\C{\textbf{C}}

\def\X{\textbf{X}}
\def\Y{\textbf{Y}}
\def\ZZ{\mathbf{Z}}

\def\0{\mathbf{0}}
\def\m{\mathfrak{m}}
\def\p{\mathfrak{p}}
\def\q{\mathfrak{q}}
\def\n{\mathfrak{n}}

\DeclareMathOperator{\charac}{char}
\DeclareMathOperator{\GL}{GL}
\DeclareMathOperator{\Kdim}{dim_{Krull}}
\DeclareMathOperator{\ima}{im}
\DeclareMathOperator{\Id}{Id}

\newtheorem*{newconj*}{\normalfont\scshape Strong Conciseness Conjecture}

\newtheorem*{newcon*}{\normalfont\scshape  Profinite Conciseness Conjecture}

\title{Conciseness in compact $R$-analytic groups}
\author{Andoni Zozaya}
\date{}

\begin{document}

\maketitle

\begin{abstract}
\blfootnote{
The author is supported by the Basque Government's project IT483-22, the Spanish Government's project PID2020-117281GB-I00, partly with ERDF funds, and the Spanish Ministry of Science, Innovation and Universities' grant FPU17/04822. \newline
\textbf{Mathematical Subject Classification (2020)}:  20E18, 20F10, 22E20 \newline
 \textbf{Keywords:} profinite groups, $R$-analytic groups, group word, conciseness, strong conciseness}
 We prove that every word is concise in the class of compact $R$-analytic groups. That is, for every word $w$ and every compact $R$-analytic group $G$ such that the set of word values $w\{G\}$ is finite, the verbal subgroup $w(G)$ is also finite. 
\end{abstract}

\section{Introduction}
\label{Section 1}

A \emph{word} in $k$ variables $x_1, \dots, x_k$ is an element $w(x_1,\dots, x_k)$ of the free group $F(x_1, \dots, x_k);$ and, given any group $G,$ it inherently defines a map $w \colon G^{(k)} \rightarrow G,$ where $w(g_1, \dots, g_k)$ is the element of $G$ that we obtain by substituting each variable $x_i$ with $g_i.$ The image of that map, which will be denoted by $w\{G\},$ is the set of \emph{word values} of $w,$ and it is typically not a subgroup of $G.$ However, we can naturally define the \emph{verbal subgroup} of $w$ as $w(G) = \langle w\{G\} \rangle.$ 

A word is said to be \emph{concise} in a class of groups $\CC$, if for all $G \in \CC$ such that $w\{G\}$ is finite, $w(G)$ is also finite. More generally, $w$ is simply called \emph{concise} if it is concise in the class of all groups. P. Hall initially conjectured that all words were concise, but Ivanov \cite{Iva} refuted it almost three decades later. Nevertheless, as Jaikin-Zapirain highlighted in \cite{verbal}, P. Hall's conjecture is still open for profinite groups:

\begin{newcon*}
Every word is concise in the class of profinite groups.
\end{newcon*}

There are several articles analysing the conciseness of certain words, but we know few classes of groups where all the words are concise. Indeed, it has been known since the 1960s that every word is concise in the class of linear groups or the class of virtually abelian-by-nilpotent groups (cf. respectively \cite{Mer} and \cite{TS}). This work provides another class of groups where every word is concise, namely compact $R$-analytic groups. In other words, we have the following:
 
\begin{mainthm*}
Every word is concise in the class of compact $R$-analytic groups.
\end{mainthm*}

\emph{$R$-analytic groups} are topological groups endowed with a manifold structure over a pro-$p$ domain $R$ such that both structures are compatible; they are defined and thoroughly studied in \cite[Chapter 13]{DDMS}, we also refer to \cite{JaKl} and \cite{Serre}. Those groups arise as a generalisation of $p$-adic analytic groups. Furthermore, every compact $p$-adic analytic groups is linear (cf. \cite[Theorem 7.19]{DDMS} and \cite[Corollary 8.33]{DDMS}), so conciseness in those groups follows from Merzljakov's result \cite{Mer}. Every word being concise in the class of compact $p$-adic analytic groups is a decisive step in Jaikin-Zapirain's study of the verbal width of finitely generated pro-$p$ groups (cf. \cite{verbal}). \\

In \cite{DMS1} and \cite{DKS}, the notion of conciseness is generalised in profinite groups. More specifically, the word $w$ is said to be \emph{strongly concise} in a class of profinite groups, if for $G \in \CC$ such that $|w\{G\}| < 2^{\aleph_0} ,$  then  $w(G)$ is finite. In \cite{DKS}, the authors pose the following conjecture:

\begin{newconj*}
Every word is strongly concise in the class of profinite groups.
\end{newconj*}

However, in the class of compact $R$-analytic groups strong conciseness is equivalent to conciseness. Indeed, in Proposition \ref{countable finite}, we show that whenever $|w\{ G\}| < 2^{\aleph_0}$ in a compact $R$-analytic group, then $w\{G\}$ is finite. Therefore, the main result can be restated as follows:
\begin{corollary}
\label{aieru1}
Every word is strongly concise in the class of compact $R$-analytic groups.
 \end{corollary}

\textit{Notation and conventions}. Most of the notation is standard, except for $X^{(n)}$, which denotes the $n$th Cartesian power of $X.$ Besides, $p$ is always a prime number, $\F_p$ is the field of $p$ elements, $\Z_p$ is the ring of $p$-adic integers, $Q[[X_1, \dots, X_n]]$ is the ring of power series in $n$ variables with coefficients in the ring $Q,$ and $\Kdim(Q)$ and $\charac{Q}$ denote, respectively, the Krull dimension and the characteristic of $Q.$  Finally,  in a topological space $X$, by $S \subseteq_o X$  we mean that $S$ is an open subset, and $\overline{S}$ denotes the closure of $S$ in $X$.

Moreover, we recall that a \emph{pro-$p$ domain} is a Noetherian local complete integral domain, whose residue field is finite of characteristic $p$. Throughout this paper $R$ will usually be a pro-$p$ domain with maximal ideal $\m,$ which sometimes will be denoted as $(R,\m)$ for clarity; additionally, this notation might be used for local rings. According to Cohen's Structure Theorem \cite{Cohen}, any pro-$p$ domain $R$ is a finite integral extension of $Q[[t_1, \dots, t_m]],$ where $m = \Kdim(R)-1,$ and $Q = \Z_p$ when $\charac{R} =0$ and $Q = \F_p[[t]]$ when $\charac{R}=p.$\\
\linebreak
\textit{Funding:} This work was supported by the Spanish Government [project number PID2020-117281GB-I00], partly with E\-R\-D\-F funds; the Basque Government [project number IT483-22] and the Spanish Ministry of Science, Innovation and Universities [grant number FPU17/04822].

\section{\texorpdfstring{$R$}{R}-standard groups}

This section is devoted to presenting and studying an essential family of $R$-analytic groups, namely $R$-standard groups. Nonetheless, for the convenience of the reader, we will start by introducing some basic definitions on the theory of $R$-analytic groups:

\begin{definition}
\label{def: analytic}
\begin{itemize}
\item[(i)] (\textup{cf. \cite[Definition 13.1]{DDMS}}). A map $f \colon U \subseteq_o R^{(n)} \rightarrow R^{(m)}$ is \emph{analytic} if for each $x \in U$ there exists a positive integer $N \in \N$ such that $x + \left( \m^N\right)^{(n)} \subseteq U$ and a tuple of power series $\FL \in R[[X_1, \dots, X_n]]^{(m)}$ such that $f(y)= \FL( x + y)$ for all $y \in \left( \m^N\right)^{(n)}.$ 
\item[(ii)] (\textup{cf. \cite[Definition 13.5]{DDMS}}). An \emph{$R$-analytic manifold} is a topological space $M$ such that for each $x \in M$ there exists an \emph{$R$-chart} $(U, \phi, n)$ of $x,$ i.e. $x \in U \subseteq_o M,$ $\phi \colon U \rightarrow \phi(U) \subseteq R^{(n)}$ is a homeomophism where $\phi(U)$ is endowed with the subspace topology and $n \in \N_0$ is the \emph{dimension}  of the chart. Moreover, all the $R$-charts must be compatible, that is, whenever $(U, \phi, n)$ and $(V, \psi, m)$ are two $R$-charts such that $U \cap V \neq \varnothing$ then $\phi \circ \psi^{-1} |_{\psi(U \cap V)} \colon \psi(U \cap V) \rightarrow R^{(n)}$ is $R$-analytic. 
\item[(iii)] (cf. \cite[Definition 13.8]{DDMS}). An \emph{$R$-analytic group} is a topological group $G$ that is an $R$-analytic manifold in such a way that the multiplication map $m \colon G \times G \rightarrow G,$ $(g,h) \mapsto g\cdot h$ and the inversion map $\iota \colon G \rightarrow G,$ $g \mapsto g^{-1}$ are analytic maps: for each $(g, h) \in G \times G$ there exist $R$-charts $(U_1, \phi_1, n_1)$ of $g$, $(U_2, \phi_2, n_2)$ of $h$ and $(V, \psi, m)$ of $gh$ such that 
$m^{-1}(V)$ is open in $G \times G$ and 
$$\psi \circ m \circ (\phi_1, \phi_2)^{-1} |_{(\phi_1, \phi_2)(U_1 \times U_2 \cap m^{-1}(V))}$$ is analytic; and the same for the inversion. 
 \end{itemize}
\end{definition}

\subsection{\texorpdfstring{$R$}{R}-standard groups and formulae}

An \emph{$R$-standard group} of level $N$ and dimension $d$ is an $R$-analytic group $S$ with a global atlas $\left\{(S, \phi)\right\}$ such that
\begin{enumerate}
\item[(i)] $\phi \colon S \rightarrow \left(  \m^N\right)^{(d)}$ is a homeomorphism,
\item[(ii)] $\phi(1) = \textbf{0}$ and 
\item[(iii)] for all $j \in \{ 1, \dots, d \}$ there exists a power series $F_j \in R[[X_1, \dots, X_{2d}]],$ with constant term equal to zero, such that  
$$\phi(xy) = \left(F_1(\phi(x), \phi(y)), \dots, F_d(\phi(x), \phi(y)) \right) \ \forall x, y \in S.$$
\end{enumerate}
Sometimes we will denote an $R$-standard group by $(S, \phi)$ to emphasise the r\^{o}le of the homeomorphism.
Any tuple of power series $\FL = (F_1, \dots, F_d)$ satisfying condition (iii) of the above definition is a $d$-dimensional \emph{formal group law}, and there exists a \emph{formal inverse} of it, i.e., a tuple of power series $\I = (I_1, \dots, I_d) \in R[[X_1, \dots, X_d]]^{(d)}$ such that
$$\FL(\I(\X), \X)  = \FL(\X, \I(\X)) = \0,$$
where $\X$ is a $d$-tuple of indeterminates (cf. \cite[Proposition 13.16(ii)]{DDMS}). 

\begin{lemma}
\label{normstd}
A compact $R$-analytic group contains an open normal $R$-standard subgroup.
\end{lemma}
\begin{proof}
Let $G$ be a compact $R$-analytic group of dimension $d$. By  \cite[Theorem 13.20]{DDMS}, there exists a finite index $R$-standard subgroup $(H, \phi)$ of level $N,$ with formal group law $\FL$ and formal inverse $\I.$ Let $T$ be a left transversal for $H$ in $G.$ Since conjugating by a fixed element is $R$-analytic (see Definition \ref{def: analytic}), for each $t \in T$ there exists an $N_t \geq N$ and some power series $C_j^t \in R[[X_1, \dots, X_d]],$ $j \in \{ 1, \dots, d\}$,  such that for all $x \in \phi^{-1}\left( \left( \m^{N_t} \right)^{(d)} \right),$
$$\phi\left(x^t\right) = \left(C_1^t(\phi(x)), \dots, C_d^{t}(\phi(x)) \right). $$
Let $L= \max_{t \in T}{N_t}$, since  $\left( \m^{L} \right)^{(d)}$ is closed with respect to the power series $F_j$ and $I_j,$ then $S = \phi^{-1}\left(\left( \m^L \right)^{(d)} \right) $ is an open $R$-standard subgroup. Moreover, $C_j^t(\0) = 0,$ so $C_j^t$ has constant term equal to zero, and thus, $S$ is closed with respect to conjugation by every $t \in T.$ Finally, since every element $g \in G$ is of the form $th$ where $t \in T$ and $h \in H,$ then $S$ is closed under the conjugation with $g.$ \end{proof} 

Furthermore, a close scrutiny of the proof reveals that for the open normal $R$-standard subgroup $(S, \phi)$ constructed in the proof, for each $g \in G$ the conjugation map $c_g \colon S \rightarrow S,$ where $c_g(x) = x^g,$  is given globally by a single tuple of power series. That is, for each $g \in G$ there exists a tuple of power series $\C_g(\X) = \left(C_{1}^{g}(\X), \dots, C_{d}^{g}(\X) \right) \in R[[X_1, \dots, X_d]]^{(d)}$ such that 
\begin{equation}
\label{conj}
\phi\left(x^g\right)  = \C_g\left(\phi(x)\right) \ \forall x \in S.
\end{equation}

Henceforward, $(S, \phi)$ will usually be an open normal $R$-standard subgroup of $G,$ of level $N$ and dimension $d,$ with formal group law $\FL$ and formal inverse $\I,$ such that the conjugation by $g \in G$ is given as in (\ref{conj}). 

Moreover, $(S, \phi)$ defines a natural atlas of $G$ given by $\{ (xS, \phi_x) \}_{x \in G},$ where $\phi_x \colon xS \rightarrow \left(  \m^N\right)^{(d)}$ is defined by $\phi_x(y) = \phi(x^{-1}y),$ and this atlas is compatible with the initial $R$-analytic structure of $G.$ We can get an explicit description of the group operations in $G$ using that atlas: 

\begin{lemma}
\label{eragiketak}
Let $t, r \in G.$ Then,
\begin{enumerate}
\item[(i)] the inverse in $tS$ is given in coordinates by the tuple of power series  $ \C_{t^{-1}} \circ \I.$  That is,
 $$\phi_{t^{-1}}\left(x^{-1} \right) = \left(\C_{t^{-1}} \circ \I \right)(\phi_t(x)) \ \forall x \in tS.$$
\item[(ii)] the multiplication in $tS \times rS$ is given in coordinates by the tuple of power series $ \FL(\C_r(\X), \Y).$ That is, 
$$\phi_{tr}(x y) =  \FL(\C_r(\phi_t(x)), \phi_r(y)) \ \forall x \in tS, y \in rS.$$
\end{enumerate}   
\end{lemma}
\begin{proof}
(i) Take $x = t\overline{x} \in tS,$ then
$$\C_{t^{-1}}(\I(\phi_t(x))) = \C_{t^{-1}}\left(\phi\left(\overline{x}^{-1}\right)\right) = \phi\left( \left(\overline{x}^{-1}\right)^{t^{-1}} \right)  = \phi_{t^{-1}}\left(x^{-1}\right).$$
(ii) Take $x = t\overline{x} \in tS$ and $y = r\overline{y} \in rS,$  then
\[ \phi_{tr}(x y) = \phi(\overline{x}^r \overline{y}) = \FL(\C_r(\phi(\overline{x})), \phi(\overline{y})) =  \FL(\C_r(\phi_t(x)), \phi_r(y)). \qedhere \]
\end{proof}

 For $x, y \in G$ such that $xS= yS,$ let $A_x^y \colon \left( \m^N \right)^{(d)} \rightarrow \left( \m^N \right)^{(d)}$ be the $R$-analytic homeomorphism $\phi_y \circ \phi_x^{-1}.$ Since $\phi_y = A_{x}^y \circ \phi_x,$ it follows from (i) of Lemma \ref{eragiketak} that
$$\phi_r\left(x^{-1}\right) = \left(A_{t^{-1}}^r \circ \C_{t^{-1}}  \circ \I \right) (\phi_t(x)) \ \forall x \in tS,$$
whenever $rS = t^{-1}S.$ And idem for the multiplication:
$$\phi_{p}(x y) =  A_{tr}^p \left( \FL(\C_r(\phi_t(x)), \phi_r(y))\right) \ \forall x \in tS, y \in rS,$$
whenever $trS =pS.$ Hence, we can fix a left transversal $T$ for $S$ in $G,$ and work simply with the atlas $\{ (tS, \phi_t) \}_{t \in T}.$

Consequently, the $R$-analytic map $w \colon G^{(k)} \rightarrow G$ is given by a single tuple of power series in the open subset $t_{1}S \times \dots \times t_{k} S$ ($t_i \in T$), i.e., there exists a tuple of power series $\W_{t_1, \dots, t_k} \in R[[X_1, \dots, X_{dk}]]^{(d)}$ such that 
\begin{equation}
\label{wordmap}
\phi_{p}(w(x_1, \dots, x_k)) = \W_{t_1, \dots, t_k} \left(\phi_{t_{1}}(x_1), \dots, \phi_{t_{k}}(x_k) \right) \ \forall x_j \in t_{j}S,
\end{equation}
where $p$ is the element of $T$ such that $w(t_1, \dots, t_k) p^{-1} \in S.$ Furthermore, $\W_{t_1, \dots, t_k}$ is a convenient composition of finitely many power series, scilicet $\FL,$ $\I,$ $\C_t$ ($t \in T \cup T^{-1}$) and $A_{x}^y$ ($x \in TT \cup T^{-1}$ and $y \in T$).

\subsection{Conciseness in \texorpdfstring{$R$}{R}-standard groups}

Since in $R$-standard groups the operations are globally described by power series, the study of the conciseness in those groups is mainly based on properties of $R$-analytic maps. 

\begin{proposition}
\label{std}
Let $S$ be an $R$-standard group and let $w$ be a word such that $w\{S\}$ is finite. Then, $w(S) = \{1\}.$ 
\end{proposition}
\begin{proof}
Firstly, $S$ can be identified with  $\left(\m^N \right)^{(d)},$ where $N$ is the level and $d$ the dimension of $S,$ such that the multiplication and the inversion are defined by two power series and the identity is $\0$, so the word map $w$ is a single power series $\mathbf{W} \in R[[X_1, \dots, X_{dk}]]^{(d)},$ where $k$ is the number of variables of the word. Moreover, since $w\{S\}$ is finite and the word map is continuous, $\W$ is locally constant, so by \cite[Lemma 9]{Ja}, $\mathbf{W}$ is constant, i.e, $\mathbf{W}(X_1, \dots, X_{dk})= \mathbf{W}(\0,\dots, \0) = \0,$ and thus $w\{S\} = \{\0\}.$
\end{proof}

We should take notice of a couple of consequences of the previous result. On the one hand, if $G$ is a compact $R$-analytic group such that $w\{G\}$ is finite, we can obtain the set of word values just by looking at a transversal of a convenient $R$-standard subgroup. 

\begin{corollary}
\label{marginal}
Let $G$ be a compact $R$-analytic group and let $S$ be an open normal $R$-standard subgroup constructed  as in Lemma \ref{normstd}. If $w\{G \}$ is finite, $S$ is marginal, i.e., 
$w(x_1, \dots, x_k) = w(y_1, \dots, y_k)$ for all $x_1, \dots, x_k, y_1, \dots, y_k \in G$ such that $x_i y_i^{-1} \in S.$ 
\end{corollary}
\begin{proof}
We shall prove that $w$ is constant in each $t_{1}S \times \dots \times t_{k}S.$ Indeed, by (\ref{wordmap}):
$$\phi_{p}(w(x_1, \dots, x_k)) = \W_{t_1, \dots, t_k}(\phi_{t_{1}}(x_1), \dots, \phi_{t_{{k}}}(x_k)) \ \forall x_j \in t_{j}S.$$
But since $w$ is locally constant, by \cite[Lemma 9]{Ja}, $\W_{t_1, \dots, t_k}$ is constant, i.e.,
$$\W_{t_1, \dots, t_k}(X_1, \dots, X_{dk})  = \mathbf{c} \in R^{(d)},$$
so $\phi_{p}(w(x_1, \dots, x_k)) = \mathbf{c}$ for all $x_j \in t_{j}S.$
\end{proof}

On the other hand, note that any $R$-analytic group $G$ satisfies a weaker version of the conciseness conjecture, namely the existence --when the set of word values $w\{G\}$ is finite-- of an open subgroup where the word $w$ is a law. 

\begin{corollary}
\label{virtually law}
Let $G$ be an $R$-analytic group and let $w$ be a word such that $w\{G\}$ is finite. Then, there exists an open $R$-standard subgroup $S$ where $w$ is a law.
\end{corollary}
\begin{proof}
According to Lemma \ref{normstd}, there exists an open $R$-standard subgroup $S$ of $G.$ Since $|w\{S\}| \leq |w\{G\}|$, by Proposition \ref{std}, it follows that $w(S) =\{1\}.$
\end{proof}

$R$-standard groups are pro-$p$ groups (cf. [6, Proposition 13.22]), and as mentioned in Section \ref{Section 1}, in this setting, the notion of conciseness can be strengthened to that of strong conciseness. Nonetheless, the following result shows that in the class of compact $R$-analytic groups there is no difference between those concepts. 

\begin{proposition}
\label{countable finite}
Let $M$ and $N$ be compact $R$-analytic manifolds and let $F \colon M \rightarrow N$ be an $R$-analytic map such that $|\ima{F}| < 2^{\aleph_0}.$ Then $\ima{F}$ is finite. 
\end{proposition}
\begin{proof}
Let $m = \dim{M}$ and $n = \dim{N}.$ Since $F$ is $R$-analytic, for each $x \in M$ there exists an $R$-chart $(U_x, \varphi)$ of $x$ in $M,$ an $R$-chart $(V_x, \psi)$ of $F(x)$ in $N$ and a tuple of power series $\G \in R[[X_1, \dots, X_m]]^{(n)}$ such that 
$$\psi \circ F \circ \varphi^{-1} (y) = \G(y) \ \forall y \in \varphi(U_x).$$
Since $(U_x, \varphi)$ is an $R$-chart, we can assume that $\varphi(U_x) = \left( \m^L \right)^{(m)}$ for some $L \in \N,$ and so $U_x$ is a profinite space. Hence, $F|_{U_x} \colon U_x \rightarrow N$ is a continuous map between profinite spaces such that $|F(U_x)| < 2^{\aleph_0},$ so, by \cite[Proposition 2.1]{DKS}, there exists $V \subseteq_o U_x$ such that $F|_V$ is constant. Hence, according to \cite[Lemma 9]{Ja}, $F|_{U_x}$ is constant. By compactness, $M = \bigcup_{z \in Z} U_z$ for some finite subset $Z \subseteq M,$ and thus $|\ima{F}| \leq |Z|$. 
\end{proof}

As a byproduct we obtain a lower bound for the cardinality of the set of word values in some finitely generated compact $R$-analytic groups. More specifically, in \cite{JaKl}, the authors present several properties of $R$-analytic groups, and among them, they establish a criterion to isolate finitely generated compact $R$-analytic groups satisfying a law, proving that they are $p$-adic analytic (here we are considering topological generation, i.e., the subgroup generated by $X$ is the topological closure of $\langle X \rangle$). 

\begin{corollary}[\textup{cf. \cite[Theorem 1.3]{JaKl}}]
Let $R$ be a pro-$p$ domain of characteristic $p$ or Krull dimension bigger than $1,$ and let $G$ be a non-discrete finitely generated compact $R$-analytic group. Then for every word $w,$ $|w\{G\}| \geq 2^{\aleph_0}.$ 
\end{corollary}
\begin{proof}
According to Lemma \ref{normstd} there exists an open $R$-standard subgroup $S$ of $G.$ Particularly, $S$ is non-discrete and finitely generated. Suppose that $|w\{G\}|< 2^{\aleph_0},$  then  $|w\{S\}| < 2^{\aleph_0}.$ Since $S$ is compact, by Proposition \ref{std} and Proposition \ref{countable finite}, $w$ is a law in $S,$ so, by \cite[Theorem 1.3]{JaKl}, $S$ admits both a $p$-adic analytic and an $R$-analytic structure, which is a contradiction with \cite[Theorem 13.23]{DDMS}.
\end{proof}

Consequently, the main theorem of this paper applies chiefly to non-finitely generated $R$-analytic groups.  

\section{Change of pro-\texorpdfstring{$p$}{p} domains}

For $\alpha = (\alpha_1, \dots, \alpha_d) \in \N_0^{(d)},$ let $\X^{\alpha}$ denote the monomial $X_1^{\alpha_1} \dots X_d^{\alpha_d}.$ Let $\varphi \colon R \rightarrow Q$ be a ring homomorphism and let $F(\X) = \sum_{\alpha \in \N_0^{(d)}} a_\alpha \X^{\alpha} \in R[[\X]]$ be a power series. By applying $\varphi$ to the coefficients of $F$ we obtain the power series 
 $$F_\varphi(\X) = \sum_{\alpha \in \N_0^{(d)}} \varphi(a_\alpha) \X^{\alpha} \in Q[[X_1, \dots, X_d]].$$
 
To transform coefficients in pro-$p$ domains, we will only consider the natural ring homomorphisms between local rings, i.e., \emph{local ring homomorphisms}, which are ring homomorphisms  $\varphi \colon (R, \m) \rightarrow (Q, \n),$ between local rings, such that $\varphi(\m) \subseteq \n.$ 

Whenever $\FL \in R[[X_1, \dots, X_n]]^{(m)}$ and $\G \in R[[X_1, \dots, X_m]]^{(l)}$ are two tuples of power series with coefficients in a pro-$p$ domain $(R, \m)$ such that $\FL(\0) \in \m^{(m)},$ the composition $\G \circ \FL$ is well-defined; and the following simple lemma shows that the foregoing change of rings commutes with the composition of power series.
 
 \begin{lemma}
 \label{konp}
Let $\varphi \colon (R,\m) \rightarrow (Q,\n)$ be a continuous local ring homomorphism. Let $\FL  \in R[[X_1, \dots, X_n]]^{(m)}$ and $\G \in R[[X_1, \dots, X_m]]^{(l)}$ be formal power series, and assume that $\FL(\0) \in \m^{(m)}.$ Then $(\G \circ \FL)_{\varphi}(X_1,\dots, X_n)=  \G_\varphi \circ \FL_\varphi(X_1, \dots, X_n).$
 \end{lemma}
 \begin{proof}
Using the universal property of power series rings (cf. \cite[Chapter 0, 7.5.3]{EGA}), there exists a unique continuous ring homomorphism 
 $$\Phi_\varphi \colon R[[X_1, \dots, X_n]] \rightarrow Q[[X_1, \dots, X_n]]$$ 
 such that $\Phi_\varphi(\mathbf{H}(\X)) = \mathbf{H}_\varphi(\X)$  for all $\mathbf{H} \in R[[\X]],$ where $\X = (X_1, \dots, X_n).$

Let $\FL(\X)= (F_1(\X), \dots, F_m(\X))$. Since $\Phi_\varphi\left(F_1(\X) \right), \dots, \Phi_\varphi\left(F_m(\X)\right)$ are in $(\n, X_1, \dots, X_n),$ the maximal ideal of $Q[[\X]]$, using the universal property of power series rings we can define the continuous map 
$$\Phi_1 \colon R[[Y_1, \dots, Y_m]] \rightarrow Q[[X_1, \dots, X_n]]$$
 such that $\Phi_1(r) = \varphi(r)$ for all $r \in R,$ and $\Phi_1(Y_i) = \left(F_i \right)_\varphi(\X) $ for all $i \in \{1, \dots, m\}.$  Similarly, define the map 
$$\Phi_2 \colon R[[Y_1, \dots, Y_m]] \rightarrow R[[X_1, \dots, X_n]]$$ 
 such that $\Phi_2|_R = \Id_R,$ and $\Phi_2(Y_i) = F_i(\X)$ for all $i \in \{1, \dots, m\}.$  

Note that $\Phi_1(r) = \Phi_\varphi \circ \Phi_2(r) = \varphi(r)$ for all $r \in R,$ and that $\Phi_1(Y_i) = \Phi_\varphi \circ \Phi_2(Y_i) = \left(  F_i \right)_\varphi(\X)$ for all $i \in \{  1, \dots, m\}.$ Therefore, by the uniqueness of the universal property we have that $\Phi_1 = \Phi_\varphi \circ \Phi_2,$ and so, in particular,
\[ \left( \G \circ \FL \right)_\varphi(\X) = \Phi_\varphi \circ \Phi_2(\G(\Y)) = \Phi_1(\G(\Y)) = \G_\varphi \circ \FL_\varphi(\X). \qedhere \]
 \end{proof} 
Hence, the change of rings preserves formal power series identities, so in particular:

\begin{corollary}
\label{FGL}
Let $R$ and $Q$ be pro-$p$ domains, let $\FL \in R[[X_1, \dots, X_{2d}]]^{(d)}$ be a formal group law with formal inverse $\I$ and let $\varphi \colon R \rightarrow Q$  be a continuous local ring homomorphism. Then $\FL_\varphi$ is a formal group law, with formal inverse $\I_\varphi.$
\end{corollary}
\begin{proof}
Let $\X,$ $\Y$ and $\ZZ$ be $d$-tuples of indeterminates. By \cite[Proposition 13.16]{DDMS}, $\FL(\0) =\0,$ so by Lemma \ref{konp},
\begin{enumerate}
\item[(i)] since 
$\FL(\FL(\X, \Y), \mathbf{Z}) = \FL(\X, \FL(\Y, \mathbf{Z}))$ then 
$$\FL_{\varphi}(\FL_{\varphi}(\X, \Y), \mathbf{Z}) = \FL_{\varphi}(\X, \FL_{\varphi}(\Y, \mathbf{Z})),$$
\item[(ii)] and since $\FL(\X, \0 ) = \FL(\0, \X) = \X ,$ then $\FL_{\varphi}(\X, \0 ) = \FL_{\varphi}(\0, \X) = \X,$ 
\end{enumerate}
and thus $\FL_\varphi$ is a formal group law (cf. \cite[Definition 13.14]{DDMS}). Finally, since $\FL(\I(\X), \X) = \FL(\X, \I(\X)) = \0$ and $\I(\0) = \0,$ by Lemma \ref{konp}:
\[ \FL_{\varphi}\left(\I_{\varphi}(\X), \X\right) = \FL_{\varphi}\left(\X, \I_{\varphi}(\X)\right) = \0. \qedhere \]  
\end{proof}

Let $(R,\m)$ and $(Q,\n)$ be pro-$p$ domains, let $\varphi \colon R \rightarrow Q$ be a continuous local ring homomorphism and let  $(S, \phi)$ be an $R$-standard group, of level $N$ and dimension $d,$ with formal group law $\FL.$ In particular, $\phi(S) = \left( \m^N\right)^{(d)}.$ 

 Using the formal group law $\FL_\varphi,$ $L = \left( \n^N \right)^{(d)}$ can be endowed with a group operation making it a $Q$-standard group. Indeed, the group operation is given by 
 \begin{equation}
 \label{Qstd}
x*y  = \FL_{\varphi}(x,y ) \  \forall x, y \in L.
\end{equation}
Let $G$ be a compact $R$-analytic group, and let $(S, \phi)$ be the open normal $R$-standard subgroup constructed in Lemma \ref{normstd}. Then, we can set up a $Q$-analytic group,  whose open normal $Q$-standard subgroup is the preceding subgroup $L.$ Indeed, let $T$ be a left transversal for $S$ in $G,$ and assume that $1 \in T.$ We will use the following notation: whenever $g \in G$ then $\tilde{g}$ is the representative of $gS$ in $T.$ By Lemma \ref{eragiketak} and the remark after it, using the notation therein, if $x \in t S$ and $y \in r S,$ their product is given by
$$\phi_{\widetilde{tr}}\left(x y \right)=  A_{tr}^{\widetilde{tr}} \left(\FL \left( \C_{r}(\phi_t(x)), \phi_r(y) \right)\right).$$
Define $H = T \times L$ and the homeomorphisms $\psi_t \colon (t, L) \rightarrow L$ such that $\psi_t(t,l) =l.$ If $x \in (t,L)$ and $y  \in (r,L)$,  imitating the previous formula define the operation:
\begin{equation}
\label{operation}
x * y = \left(\widetilde{tr}, \left(A_{tr}^{\widetilde{tr}}\right)_\varphi \left( \FL_{\varphi}\left( \left(\C_{r} \right)_{\varphi}(\psi_t(x)), \psi_r(y) \right) \right)\right). 
\end{equation}
\begin{remark*}
We can identify $(1, L)$ with $L$ for plainness, and then (\ref{operation}) extends the operation (\ref{Qstd}).
\end{remark*}
\begin{lemma}
\label{Qazpitalde}
With the notation above, $(H, *)$ is a compact $Q$-analytic group with open normal $Q$-standard subgroup $L.$
\end{lemma}
\begin{proof}
Firstly, $H$ is a compact $Q$-analytic manifold with respect to the atlas $\{ (tL, \psi_t) \}_{t \in T}$ -- abusing the notation we will use $tL$ to denote $(t,L)$ --. 
Moreover, $H$ is a group. Indeed,\\
\linebreak
(i) let $t,r,p \in T,$ from Lemma \ref{eragiketak} and the associativity of $G$ we know that as power series:
\begin{multline*}
A_{t \cdot \widetilde{rp}}^{\widetilde{trp}}\left( \FL\left( \C_{\widetilde{rp}}(\X), A_{rp}^{\widetilde{rp}}(\FL(\C_p(\Y), \ZZ)) \right) \right) = \\A_{\widetilde{tr}\cdot p}^{\widetilde{trp}}\left( \FL\left( \C_p\left( A_{tr}^{\widetilde{tr}}\left( \FL\left( \C_r(\X), \Y \right)\right) \right), \ZZ \right) \right).
\end{multline*}
Thus, let $x \in tL,$ $y \in rL$ and $z \in pL,$ then by Lemma \ref{konp} and (\ref{operation}): 
\begin{small}
\begin{align*}
\MoveEqLeft[3] \psi_{\widetilde{trp}}(x*(y*z)) & \\&=  \left(A_{t \cdot \widetilde{rp}}^{\widetilde{trp}}\right)_{\varphi} \left( \FL_\varphi\left( \left(\C_{\widetilde{rp}}\right)_{\varphi}(\psi_t(x)), \left(A_{rp}^{\widetilde{rp}}\right)_{\varphi}\left(\FL_\varphi\left(\left(\C_{p}\right)_{\varphi}(\psi_r(y)), \psi_p(z)\right)\right) \right) \right) \\&
= \left(A_{\widetilde{tr}\cdot p}^{\widetilde{trp}} \right)_{\varphi}\left( \FL_\varphi\left( \left(\C_{p}\right)_{\varphi}\left( \left(A_{tr}^{\widetilde{tr}} \right)_{\varphi}\left( \FL_\varphi\left( \left(\C_{r}\right)_{\varphi}(\psi_t(x)), \psi_r(y) \right)\right) \right), \psi_p(z) \right) \right) 
 \\& = \psi_{\widetilde{trp}}((x*y)*z).
 \end{align*}
 \end{small}
(ii) The neutral element is $ (1,\0) \in L.$ \\
\linebreak
(iii) The inverse of $x \in t L$ is given by 
$$ y = \left(\widetilde{t^{-1}},  \left(A_{t^{-1}}^{\widetilde{t^{-1}}}\right)_\varphi \circ \left(\C_{t^{-1}}\right)_{\varphi} \circ \I_{\varphi}(\psi_t(x)) \right).$$
Indeed, by Lemma \ref{eragiketak}, we know that
$$A_{t\cdot \widetilde{t^{-1}}}^1\left( \FL\left( \C_{\widetilde{t^{-1}}}(\X), A_{t^{-1}}^{\widetilde{t^{-1}}}\left( \C_{t^{-1}}(\I(\X)) \right)  \right) \right) = \0.$$
Hence, by Lemma \ref{konp} and (\ref{operation}),
\begin{align*}
 \0&= 
\left(A_{t \cdot \widetilde{t^{-1}}}^1\right)_{\varphi}\left( \FL_\varphi\left( \left(\C_{\widetilde{t^{-1}}}\right)_{\varphi}(\psi_t(x)), \left(A_{t^{-1}}^{\widetilde{t^{-1}}}\right)_{\varphi}\left( \left(\C_{t^{-1}} \right)_{\varphi}(\I_\varphi(\psi_t(x))) \right)  \right) \right)\\ & 
= \left(A_{t\cdot \widetilde{t^{-1}}}^1\right)_{\varphi}\left( \FL_\varphi\left( \left(\C_{\widetilde{t^{-1}}}\right)_{\varphi}(\psi_t(x)), \psi_{\widetilde{t^{-1}}}(y)  \right) \right)\\& 
=\psi_1( x * y).
 \end{align*}
 Similarly, $y*x  = (1,\0). $\\
 
 Finally, if $x \in L$ and $h \in tL \subseteq H,$ since $1^t =1,$ then $x^{*h} \in L.$ 
\end{proof}

Let $(t_1, \dots, t_k) \in T^{(k)}$. According to (\ref{wordmap}), the word map $w\colon G^{(k)} \rightarrow G$ is given in $t_1 S \times \dots \times t_kS$ by the power series $\W_{t_1, \dots, t_k},$ i.e.,
$$\phi_{\widetilde{w(t_1, \dots, t_k)}} (w(x_1, \dots, x_k)) = \W_{t_1, \dots, t_k} \left( \phi_{t_1}(x_1), \dots, \phi_{t_k}(x_k) \right) \ \forall x_i \in t_iS.$$
Then, by the definition of $F_{\varphi},$ $\I_\varphi$ and (\ref{operation}), the word map $w \colon H^{(k)} \rightarrow H$ is given in $t_1 L \times \dots \times t_kL$ by the power series $\left(\W_{t_1, \dots, t_k}\right)_{\varphi},$ i.e,
 \begin{equation}
 \label{genword}
 \psi_{\widetilde{w(t_1, \dots, t_k)}} (w(x_1, \dots, x_k)) = \left(\W_{t_1, \dots, t_k}\right)_{\varphi}\left( \psi_{t_1}(x_1), \dots, \psi_{t_k}(x_k) \right) \ \forall x_i \in t_iL.
 \end{equation}

\section{Proof of the Main Theorem}

The demonstration technique consists in reducing the problem to a group that is analytic over a pro-$p$ domain of Krull dimension one, using the change of rings described in the previous section. Hence, we have first to prove the 1-dimensional case.

\subsection{Pro-\texorpdfstring{$p$}{p} domains of Krull dimension one}

We shall start with a well-known but useful observation, whose proof can be found, for instance, in \cite[Lemma 4]{DMS}.
\begin{lemma}
\label{abelian}
Let $G$ be a group and let $w$ be a word such that $w\{G\}$ is finite. Then $w(G)'$ is finite, and $w(G)$ is finite if and only if it has finite exponent. 
\end{lemma}

We have already pointed out that compact $p$-adic analytic groups are concise. By Cohen's Structure Theorem (loc. cit.) and \cite[Theorem 13.23]{DDMS}, the result extends to analytic groups over Krull dimension one and characteristic zero pro-$p$ domains.  Moreover, whenever $R$ is a pro-$p$ domain of Krull dimension one and characteristic $p,$ according to Cohen's Structure Theorem (loc. cit.) $R$ is a finitely generated  free $\F_p[[t]]$-module, so by \cite[Examples 13.9(iv)]{DDMS}, any $R$-analytic group is by restriction of scalars analytic over $\F_p[[t]].$ Note that in this case, any $R$-analytic group has also analytic structure over the local field $\F_p((t))$ (cf. \cite[Section 13.1]{DDMS}).

\begin{proposition}
\label{PID}
Let $R$ be a pro-$p$ domain of Krull dimension $1$. Then, every word is concise in the class of compact $R$-analytic groups.
\end{proposition}

\begin{proof}
We only have to deal with the case of positive characteristic. Let $G$ be a compact $R$-analytic group, let $w$ be a word such that $w\{G\}$ is finite, let $S$ be a normal open $R$-standard subgroup of $G$ where $w$ is a law (compare with Corollary \ref{virtually law}), $Z= Z(S),$ $K = \F_p((t))$ and $\bar{K}$ the algebraic closure of $K$. We can assume that $G$ is $K$-analytic.  

Firstly, by \cite[Proposition 5.1]{Jai} $S/Z$ is linear over the local field $K$ and it satisfies an identity, so by the topological Tits alternative (cf. \cite[Theorem 1.3]{BG}) $S/Z$ has an open soluble subgroup. Let $\mathcal{S}$ be the Zariski closure of $S/Z$ in $\GL_n(\bar{K}),$ which is also virtually soluble (cf. \cite[Theorem 5.11]{Weh}). Let  $\mathcal{S}^\circ$ be identity component of $\mathcal{S}$, which according to  \cite[Proposition 7.3]{Hump} has finite index in $\mathcal{S}$ and  is the soluble radical of $\mathcal{S},$ namely the largest connected soluble subgroup of the algebraic group $\mathcal{S}.$
 
Let $N/Z$ be the intersection of $S/Z$ with $\mathcal{S}^\circ,$ then $N$ has finite index in $S$ and passing to the normal core if necessary we assume that $N$ is a normal subgroup of finite index in $G$.  Since $[\mathcal{S}^{\circ}, \mathcal{S}^\circ]$ is unipotent (cf. \cite[Lemma 19.5]{Hump}), and since we are in characteristic $p,$ $[N, N] Z/Z$ is nilpotent of finite exponent. Therefore, $ [N,N]Z$ is nilpotent, and, according to \cite[Lemma 5.1]{JaKl}, $H= [N,N ,S]$ has finite exponent. Finally, 
$$H = [N,N,S] \geq [N,N,N],$$
and so $G/H$ is virtually nilpotent of class at most $2.$   

Furthermore, since $w(G)'$ is finite by Lemma \ref{abelian}, up to a quotient we can assume that $w(G)$ is a finitely generated abelian subgroup. Hence, since $H$ has finite exponent, $w(G) \cap H$ is finite.  

Finally, since $|w\{ G/H\}| \leq |w\{G\}|$ and  $G/H$ is virtually nilpotent, then $w\left( G/H \right)$ is finite by \cite[Corollary 2]{TS}. Now, the isomorphism
$$w\left( \frac{G}{H} \right) = \frac{w(G)H}{H} \cong \frac{w(G)}{w(G) \cap H}$$
yields the result. 
\end{proof}

Note that, as a byproduct, we have proved that any soluble $R$-analytic group, where $R$ has characteristic $p,$ is virtually (nilpotent of finite exponent)-by-(nilpotent of class at most $2$).

\subsection{General case}

When $(R, \m)$ is a pro-$p$ domain, for any $a \in \m^{(m)}$ we have the continuous \emph{evaluation map} $s_a \colon R[[t_1, \dots, t_m]] \rightarrow R$ which sends $F(t_1, \dots, t_m)$ to $F(a).$ As the following classic result shows, this local ring epimorphism can be extended to any finite integral extension of $R[[t_1, \dots, t_m]]$. 

\begin{lemma}[Going Up Theorem]
\label{GUT}
Let $A \subseteq B$ be a finite integral extension of rings, let $P$ be an integral domain and let $ \varphi \colon A \rightarrow P$ be a ring epimorphism. There exists a finite integral extension $Q$ of $P$ such that $\varphi$ extends to an epimorphism $\tilde{\varphi} \colon B \rightarrow Q.$
 \end{lemma}
 \begin{proof}
 If $\p = \ker{\varphi},$ by \cite[Theorem V.2.3]{ZaSa}, there exists a prime ideal $\q \subseteq B$ such that $\q \cap A = \p.$ Thus, the following diagram is commutative:
$$\begin{tikzcd}
B \arrow[r, "\tilde{\varphi}"]                 & B/\q                                      \\
A \arrow[r, "\varphi"] \arrow[u, hook] & A/\p \arrow[u, "\psi"] 
\end{tikzcd}$$
where $\psi(x  + \p) = x +\q$ is injective. Identifying $A/\mathfrak{p}$ with $P,$ then $\tilde{\varphi}$ extends $\varphi$ and, by \cite[Lemma V.2.1]{ZaSa}, $B/\q$ is a finite integral extension of $A/\p.$ 
 \end{proof}
 \begin{remark*}
In the preceding proof, if $B$ is a pro-$p$ domain, $Q$, as a quotient of a pro-$p$ domain by a prime ideal, is itself a pro-$p$ domain; and, when $\varphi \colon A \rightarrow P$ is a continuous map between pro-$p$ domains, then $A/\p$ is isomorphic to $P$ as a topological ring.
 \end{remark*}
 
 \begin{proof}[Proof of the Main Theorem]
 
Let $G$ be a compact $R$-analytic group. By Lemma \ref{normstd}, there exists an open normal $R$-standard subgroup $(S, \phi),$ of level $N,$ such that, for each $g \in G$ there exists a tuple of power series $\C_g \in R[[X_1, \dots, X_d]]^{(d)}$ satisfying
$$\phi(x^g) = \mathbf{C}_g(\phi(x)) \ \forall x \in S.$$

Let $w$ be a word in $k$ variables such that $w\{G\}$ is finite. If $n = |G:S|$, then $w^n\{ G \} \subseteq S$ and, by Lemma \ref{abelian}, if $w^n(G) $ is finite, then $w(G)$ is finite. Therefore, without loss of generality assume that $w\{G\} \subseteq S.$ 
 
Furthermore, let $P$ be  $\Z_p$ if $\charac{R} = 0$ or $\F_p[[t]]$ if $\charac{R} =p,$ and $\m$ its maximal ideal. According to Cohen's Structure Theorem (loc. cit.) $R$ is a finite integral extension of $P[[t_1, \dots, t_{m}]],$ where $m =\Kdim(R) -1$.  For each $a \in \m^{(m)},$ let $s_a$ be the evaluation epimorphism $s_a \colon P[[t_1, \dots, t_{m}]] \rightarrow P.$ By Lemma \ref{GUT}, $s_a$ extends to a continuous ring epimorphism $\tilde{s}_a \colon R \rightarrow Q,$ where $Q$ is a pro-$p$ domain of Krull dimension $1.$

Let $\FL$ be the formal group law of $S.$ Fix $a \in \m^{(m)}$ and denote by $\FL_a$ the formal group law $\FL_{\tilde{s}_a}$ (cf. Corollary \ref{FGL}). Let $T$ be a left transversal for $S$ in $G,$ and  assume that $1 \in T$. According to Lemma \ref{Qazpitalde}, there exists a $Q$-standard group $L,$ whose group operation is given by $\FL_a,$ and a compact $Q$-analytic group $H,$ which is an overgroup of $L$ and whose group operation is defined as in (\ref{operation}). Moreover, the $Q$-analytic structure of $H$ is given by the atlas $\{ (tL, \psi_t) \}_{t \in T},$ where $\psi_t(t,l) =l$.

Fix $(t_{1}, \dots, t_{k}) \in T^{(k)}$ and assume, by  (\ref{wordmap}), that for any $l \in \N,$ the word map $w^{l}$ is given in $t_{1} S \times \dots \times t_{k} S$ by the power series $\W^{l},$ that is, 
$$\phi\left(w^{l}(x_1, \dots , x_k)\right) = \W^{l}\left(\phi_{t_{1}}(x_{1}), \dots, \phi_{t_{k}}(x_{k}) \right) \ \forall x_{j} \in t_{j}S$$
(even though in order to lighten the notation it is not written explicitly, the power series $\W^l$ also depends on $t_1, \dots, t_k$).

Let $w^l \colon H^{(k)} \rightarrow H$ be the word map $w^l$ with the operation of $H$ induced by $\FL_a,$ as in (\ref{operation}). Thus, by (\ref{genword}),
$$\psi_1\left(w^l(x_1, \dots, x_k)\right) = \W^{l}_{a}\left(\psi_{t_{1}}(x_{1}), \dots, \psi_{t_{k}}(x_{k})\right) \ \forall x_{j} \in t_{j}L.$$
(for any power series $\W,$ use $\W_a$ to denote $\W_{\widetilde{s}_a}$).\\

Therefore, according to Corollary \ref{marginal}, if $w\{G\}$ is finite, then $w$ is constant on each $t_{1} S \times \dots \times t_{k}S.$ Hence, the word map $w,$ with the operation of $H,$ is constant on each $t_{1}L \times \dots \times t_{k}L,$ so $|w\{H\}| \leq |H:L|^{k}$ is finite. According to Proposition \ref{PID}, every word is concise in $H$, so there exists $\ell_a \in \N$ such that $w^{\ell_a}(H) = \{ 1 \},$ and, in particular, 
\begin{equation}
\label{la}
\W_{a}^{\ell_a}(X_1, \dots,X_{dk}) = \0.
\end{equation}
Define the following partition of the space $\m^{(m)}$: 
$$\m_\ell =  \left\{ a \in \m^{(m)} \middle\vert \W_a^{\ell} = \0 \right\}, \ \ell \in \N.$$

Since $\m^{(m)} = \bigcup_{\ell \in \N} \overline{\m_\ell}$ and $\m^{(m)}$ is complete, by the Baire Category Theorem there exists $\ell'$ such that $\overline{\m_{\ell'}}$ has non-empty interior. Moreover, by Corollary \ref{marginal},$\W^{\ell'}$ is constant, i.e.,
$$\W^{\ell'}(X_1, \dots, X_{dk}) = (c_1, \dots, c_d) \in R^{(d)}.$$
Then, if $a \in \m_{\ell'},$
$$\left(\tilde{s}_a(c_1), \dots, \tilde{s}_a(c_d) \right)  = \W_{a}^{\ell'}(X_1, \dots, X_{dk}) = \0.$$
Consequently, $\tilde{s_a}(c_i) = 0$ for all $a \in \m_{\ell'}.$ 

\begin{subclaim*}
Let  $V \subseteq_o \m^{(m)}$ and $D$ a dense subset of $V,$ then $\bigcap_{a \in D} \ker{\tilde{s}_a} = \{0\}.$ 
\end{subclaim*}
\begin{proof}[Proof of the subclaim]
Let $r \in P[[t_1, \dots, t_m]],$ since the evaluation of a power series is continuous, if $r(a) = 0$ for all $a \in D,$ then $r(a) = 0$ for all $a \in V.$ Hence, by \cite[Lemma 9]{Ja}, we have $\bigcap_{a \in D} \ker{s_a}= \{0\}.$  If $\p_a =\ker{s_a}$ and  $\q_a = \ker{\tilde{s}_a},$ then $\p_a = \q_a \cap P[[t_1, \dots, t_m]].$ Hence,
$$P[[t_1, \dots, t_m]] \cap \bigcap_{a \in D} \q_a = \bigcap_{a \in D} \p_a = \{ 0\},$$
and, since the extension is integral, according to \cite[Complement 1 to Theorem V.2.3]{ZaSa}, we have $\bigcap_{a \in D} \q_a = \{ 0\}.$
\end{proof}

By the subclaim, applied to the interior of $\overline{\m_{\ell'}}$ as $V$ and $\m_{\ell'} \cap V$ as $D,$  we have that $c_i = 0$ for all $i \in \{1, \dots, d \}.$ Therefore, repeating the process for all the tuples in $T^{(k)}$ there exists an integer $\ell$ such that 
$\phi\left(w^\ell(x_1, \dots, x_k)\right) = \0$ for all $x_i \in G$, and thus $w^\ell(G) = \{1 \}.$
\end{proof}

\begin{remark*}
It is known that uniform pro-$p$ groups are torsion-free (cf. \cite[Theorem 4.5]{DDMS}), and since every word is concise in them, whenever $H$ is a compact $p$-adic analytic group such that $w\{H\}$ is finite, and $L$ is an open uniform pro-$p$ subgroup, we have that $w\{H\}^{|H:L|} =\{1\}.$ Using this fact, the proof of the Main Theorem can be slightly simplified when $R$ is a pro-$p$ domain of characteristic zero. Indeed, in (\ref{la}), we can take $\ell_a = |G:S|,$ independently on the evaluation $\tilde{s}_a$, and thus avoid the use of Baire's Category Theorem --here we follow the notation in the proof, that is, $G$ is a compact $R$-analytic group and $S$ is the open normal $R$-standard subgroup constructed as in Lemma \ref{normstd}--. Note that, in passing, it would prove that every element of $w\{G\}$ has order dividing $|G:S|$.
\end{remark*}

\section*{Acknowledgements}
The author would like to thank his thesis advisors, Gustavo A. Fern\'andez-Alcober and Jon Gonz\'alez-S\'anchez, for their helpful suggestions and detailed feedback on earlier versions of this paper.


\begin{thebibliography}{apalike}

    \bibitem{BG}
    Breuillard,~E. and  Gelander,~T.: A topological Tits alternative,  \textit{Ann. of Math.} \textbf{166}, 427--474 (2007).
    
        \bibitem{Cohen}
    Cohen,~I.~S.: On the structure and ideal theory of complete local rings,  \textit{Trans. Amer. Math. Soc.} \textbf{59}, 54--106 (1946).

   \bibitem{DKS}
Detomi,~E., Klopsch,~B. and Shumyatsky,~P.: Strong conciseness in profinite groups,  \textit{J. London Math. Soc.} \textbf{102}, 977--993 (2020). 

   \bibitem{DMS1}
Detomi,~E., Morigi,~M. and Shumyatsky,~P.: On conciseness of words in profinite groups,  \textit{J. Pure Appl. Algebra} \textbf{220}, 3010--3015 (2016).

   \bibitem{DMS}
Detomi,~E., Morigi,~M. and Shumyatsky,~P.: Words of Engel type are concise in residually finite groups,  \textit{Bull. Math. Sci.} \textbf{9}, 1950012 (2019).

\bibitem{DDMS}
 Dixon,~J.,  Du Sautoy,~M., Mann,~A. and  Segal,~D.: Analytic Pro-$p$ Groups, 2nd ed., Cambridge University Press, Cambridge (1999).
 
            \bibitem{EGA}
   Grothendieck,~A. and Dieudonn\'e,~J.: \'El\'ements de geom\'etrie alg\'ebrique I. Le langage des sch\'emas,  \textit{Publ. Math. IH\'ES} \textbf{4}, 5--228 (1960).
 
                     \bibitem{Hump}
     Humphreys,~J.~E.: Linear Algebraic Groups, Springer-Verlag, New York-Berlin-Heidelberg (1998). 
    
           \bibitem{Iva}
   Ivanov,~S.~V.: P. Hall's conjecture on the finiteness of verbal subgroups,  \textit{Sov. Math.} \textbf{33}, 59--70 (1989).

              \bibitem{Jai}
   Jaikin-Zapirain,~A.: On linear just infinite pro-$p$ groups,  \textit{J. Algebra} \textbf{255}, 392--404 (2002).
   
      \bibitem{Ja}
    Jaikin-Zapirain,~A.: On linearity of finitely generated $R$-analytic groups,  \textit{Math. Z.} \textbf{253}, 333--345 (2006).
  
               \bibitem{verbal}
   Jaikin-Zapirain,~A.: On the verbal width of finitely generated pro-$p$ groups,  \textit{Rev. Mat. Iberoam.} \textbf{24}, 617--630 (2008). 
    
   \bibitem{JaKl}
    Jaikin-Zapirain,~A. and Klopsch,~B.: Analytic groups over general pro-$p$ domains,  \textit{J. London Math. Soc.} \textbf{76}, 365--383 (2007).

      \bibitem{Mer}
      Merzljakov,~Yu.~I.: Verbal and marginal subgroups of linear groups,  \textit{Dokl. Akad. Nauk SSSR} \textbf{177}, 1008--1011 (1967).

    \bibitem{Serre}
     Serre,~J.~P.: Lie Algebras and Lie Groups,  Lecture Notes in Mathematics vol. \textbf{1500}, Springer-Verlag, Berlin-Heidelberg (2006). 
      
                       \bibitem{TS}
  Turner-Smith,~R.~F.: Finiteness conditions for verbal subgroups,  \textit{J. London Math. Soc.} \textbf{41}, 166--176 (1966).
      
                   \bibitem{Weh}
      Wehrfritz,~B.~A.~F.: Infinite Linear Groups, Springer-Verlag, Berlin-Heidelberg-New York (1973).
      
       \bibitem{ZaSa}
      Zariski,~O. and Samuel,~P.: Commutative Algebra, Vol. I and II, Springer-Verlag, Berlin-Heidelberg-New York (1958).
       
\end{thebibliography}
\end{document}